\pdfoutput=1

\documentclass[10pt]{amsart}
\usepackage{microtype}
\usepackage[utf8]{inputenc}

\usepackage{amsmath,amstext,amssymb,mathrsfs,amscd,amsthm,indentfirst, bbm}
\usepackage{amsfonts}

\usepackage{todonotes}

\usepackage[style=ieee-alphabetic, url=false, backref=true, doi=false, isbn=false, eprint=false]{biblatex}
\usepackage{float}
\usepackage{adjustbox}
\usepackage{booktabs}
\usepackage{verbatim}
\usepackage{graphicx}
\usepackage{textcomp}
\usepackage{tikz,tikz-cd}
\usetikzlibrary{patterns}
\usepackage{ifthen}
\usepackage{url}
\usepackage{hyperref}
\PassOptionsToPackage{hyphens}{url}\usepackage{hyperref}
\hypersetup{
    colorlinks  = true,
    linkcolor   = [rgb]{.20,.60,.22},
    citecolor   = [rgb]{0,.40,.80},
    urlcolor    = [rgb]{0,.40,.80}
}
\usepackage[nameinlink,capitalize]{cleveref}
\usepackage{mathtools}
\usepackage{mleftright}
\usepackage{enumitem}

\usepackage{baskervald}
\usepackage[hmargin=3.3cm,vmargin=3.5cm]{geometry}
\newcommand{\bC}{\mathbb C}

\newcommand{\bN}{\mathbb N}
\newcommand{\bP}{\mathbb P}

\newcommand{\bZ}{\mathbb Z}

\newcommand{\cA}{\mathcal A}
\newcommand{\cB}{\mathcal B}
\newcommand{\cC}{\mathcal C}

\newcommand{\cE}{\mathcal E}

\newcommand{\cI}{\mathcal I}

\newcommand{\cK}{\mathcal K}
\newcommand{\cL}{\mathcal L}

\newcommand{\cN}{\mathcal N}
\newcommand{\cO}{\mathcal O}
\newcommand{\cP}{\mathcal P}

\newcommand{\cZ}{\mathcal Z}

\newcommand{\lra}{\longrightarrow}

\newcommand{\Gammahol}{\Gamma_{\mathrm{hol}}}
\newcommand{\Gammacon}{\Gamma_{\cC^0}}
\newcommand{\Gammaconfib}{\Gamma_{\cC^0, \mathrm{fib}}}

\newcommand{\Pic}{\mathrm{Pic}}

\newcommand{\Hol}{\mathrm{Hol}}
\newcommand{\Map}{\mathrm{Map}}
\newcommand{\Hom}{\mathrm{Hom}}
\newcommand{\Alg}{\mathrm{Alg}}

\DeclareMathOperator{\unrelProj}{Proj}
\newcommand{\Proj}{\underline{\unrelProj}}
\DeclareMathOperator{\Sym}{Sym}
\DeclareMathOperator{\Spec}{Spec}
\newcommand{\NS}{\mathrm{NS}}

\newcommand{\catSch}{\mathsf{Sch}^\mathrm{ft}_\bC}
\newcommand{\catSet}{\mathsf{Set}}

\newcommand{\opp}{\mathrm{op}}
\DeclareMathOperator{\rel}{rel}

\newcommand{\defeq}{\mathrel{\vcentcolon=}}
\NewDocumentCommand{\set}{somm}{%
   \IfNoValueTF{#2}
    {\IfBooleanTF{#1}{\{#3 \mid #4\}}{\mleft\{ #3 \mathrel{}\middle\vert\mathrel{} #4 \mright\}}}
    {\mathopen{#2\{}#3 \mathrel{}#2\vert\mathrel{} #4\mathclose{#2\}}}%
  }

\newtheorem{theorem}{Theorem}[section]
\newtheorem*{theorem*}{Theorem}
\newtheorem{theorem-in-progress}[theorem]{Theorem-in-progress}

\newtheorem*{conjecture*}{Conjecture}

\newtheorem{definition}[theorem]{Definition}
\newtheorem{lemma}[theorem]{Lemma}
\newtheorem{remark}[theorem]{Remark}
\newtheorem{proposition}[theorem]{Proposition}
\newtheorem{example}[theorem]{Example}

\newtheorem{maintheorem}{Theorem}

\newcommand{\stacks}[1]{\cite[\href{https://stacks.math.columbia.edu/tag/#1}{Tag #1}]{stacks-project}}

\title{The topology of spaces of holomorphic maps to projective space}

\makeatletter         
\renewcommand\maketitle{
{\raggedright
\begin{center}
    {\large\bfseries\MakeUppercase\@title}\\[0.6em]
    {\scalebox{0.8}{BY}}\\[0.6em]
    {\small\MakeUppercase{Alexis Aumonier}}\\[0.6em]
    {\footnotesize\textit{Cambridge, England}}
\end{center}
}
}
\makeatother

\address{Center for Mathematical Sciences, Wilberforce Road, Cambridge CB3 0WB, UK}
\email{aa2099@cam.ac.uk}

\begin{document}

\maketitle

\section{Introduction}

Let $X$ be a connected smooth complex projective variety of dimension $n \geq 1$, and let $\bP^m$ denote the complex projective space of dimension $m \geq n$. Recall that $H^2(\bP^m;\bZ) \cong \bZ \cdot h$ is a free abelian group generated by an element $h$ Poincar\'e dual to a hyperplane. We define the degree of a map $f \colon X \to \bP^m$ to be the cohomology class $f^*(h) \in H^2(X;\bZ)$. In this paper we prove:
\begin{maintheorem}\label{thm:main-intro}
Let $\alpha \in H^2(X;\bZ)$ be such that $\alpha - c_1(K_X)$ is ample\footnote{We write $K_X$ for the canonical line bundle, and recall that ampleness is a numerical property by the Nakai--Moishezon criterion.}. Then the inclusion
\[
    \Hol_\alpha(X,\bP^m) \hookrightarrow \Map_\alpha(X,\bP^m),
\]
of the subspace of holomorphic maps of degree $\alpha$ inside the space of continuous maps of the same degree, induces an integral homology isomorphism in the range of degrees $* < (2m-2n+1) \cdot d(\alpha) + 2(m-n-1)$, for some number $d(\alpha)$ defined in \cref{def:dalpha-range}.
\end{maintheorem}

When $X$ is simply connected and $m > n$, both mapping spaces are simply connected and the conclusion of the theorem can be improved to a homotopy equivalence in the same range by the homology Whitehead theorem. We shall explain later how to estimate the number $d(\alpha)$, but let us for now only give some examples:
\begin{enumerate}
    \item If $X$ is a Riemann surface of genus $g$ and $\alpha \in \bZ \cong H^2(X;\bZ)$ denotes the classical degree, then $d(\alpha) \geq \alpha - 2g$ if $\alpha \geq 2g$.
    \item If $X = \bP^n$ and $\alpha \in \bZ \cong H^2(\bP^n;\bZ)$ is the classical degree, then $d(\alpha) \geq \alpha$ if $\alpha \geq 0$.
    \item If $X = \bP^{n_1} \times \cdots \times \bP^{n_r}$ is a product of projective spaces, a map has a degree $\alpha = (\alpha_1,\dotsc,\alpha_r) \in \bZ^r$ and $d(\alpha) \geq \min (\alpha_1,\dotsc,\alpha_r)$.
    \item For any $X$ and for any $M > 0$, there exists $\alpha \in H^2(X;\bZ)$ such that $d(\alpha) \geq M$. In particular the range of degrees in our theorem can be arbitrarily big.
\end{enumerate}

\subsection{Relations to previous work}

In his seminal paper \cite{segal_topology_1979}, Segal proved \Cref{thm:main-intro} when $X$ is a Riemann surface. In the forty years that followed, much progress was made to extend his result to more general targets, including, to cite a few, Grassmannians \cite{kirwan_spaces_1985}, flag manifolds \cites{guest_topology_1984,mann_moduli_1993,boyer_topology_1994,hurtubise_holomorphic_1996,boyer_stability_2001}, loop groups \cite{gravesen_topology_1989}, and toric varieties \cite{guest_topology_1995}. However little headway was made in cases where the domain $X$ has dimension greater than one. A notable exception is Mostovoy's proof of \cref{thm:main-intro} when $X = \bP^n$ \cite{mostovoy_spaces_2006}. The problem of generalising Segal's theorem to higher dimensional varieties has thus attracted some attention, see e.g. \cite[Question~29]{gomez-gonzales_problems_2020}. We propose an approach in this paper.

In a somewhat opposite direction, one could wonder about analogues of \cref{thm:main-intro} when $X$ is not projective. Then the results in the literature have a very different flavour. Indeed, when $X$ is Stein, the Oka principle \cite{gromov_okas_1989} states that the inclusion $\Hol(X,\bP^m) \hookrightarrow \Map(X,\bP^m)$ is a weak homotopy equivalence. However such a strong conclusion cannot possibly hold when $X$ is projective, as seen for example by noting that the homolorphic mapping space is finite dimensional whereas the continuous one might have infinite cohomological dimension.

\subsection{Open questions}

When $m = n$, the mapping spaces are not simply-connected anymore, and the homology Whitehead theorem cannot be applied directly. The space $\Map_\alpha(X,\bP^m)$ is always nilpotent, and Segal shows in \cite[Corollary~6.3]{segal_topology_1979} that $\Hol_\alpha(\bP^1,\bP^1)$ is ``nilpotent in a range'' which implies in that case that the inclusion induces an isomorphism on homotopy groups in a range. It would be interesting to know if this occurs in other cases.

In \cite{kirwan_spaces_1985}, Kirwan uses Segal's theorem to study holomorphic maps from Riemann surfaces to Grassmannians. This has notably applications in holomorphic K-theory, see e.g. \cite[Section~7]{friedlander_semi-topological_2002}. It is natural to ask about generalisations of Kirwan's results to higher dimensional varieties, although a naive transcription of her arguments faces the formidable problem of understanding vector bundles on higher dimensional varieties.

\subsection{Proof strategy}

A map $X \to \bP^m$ is expressed in homogeneous coordinates by a formula of the form
\[
    X \ni x \longmapsto [s_0(x):\dotsc:s_m(x)] \in \bP^m,
\]
where $s_i$ are global sections of a line bundle $\cL$ on $X$, either holomorphic or continuous, that never vanish simultaneously. Thus, when $c_1(\cL) = \alpha$, and assuming that $X$ is simply connected for simplicity, the inclusion
\[
    \Hol_\alpha(X,\bP^m) \cong \Gammahol(\cL^{\oplus m+1} - 0) / \bC^\times \hookrightarrow \Map_\alpha(X,\bP^m) \cong \Gammacon(\cL^{\oplus m+1} - 0) / \bC^\times
\]
is a homology equivalence in a range whenever it is the case for the inclusion
\[
    \Gammahol(\cL^{\oplus m+1} - 0) \hookrightarrow \Gammacon(\cL^{\oplus m+1} - 0)
\]
of the holomorphic section space inside the continuous one. This was shown in our earlier work \cite{aumonier_h-principle_2021} and directly implies \cref{thm:main-intro} when $X$ is simply connected\footnote{In that case we even have the slightly better range $* < (2m-2n+1)\cdot d(\alpha) + 2(m-n)$, see \cref{thm:hprinciple}. The difference of $-2$ is an artefact of our proof of the general case.}. In general, when $H^1(X;\bZ) \neq 0$ we only have
\[
    \Map_\alpha(X,\bP^m) \cong \Gammacon(\cL^{\oplus m+1} - 0) / \Map(X,\bC^\times)
\]
and
\[
    \Hol_\alpha(X,\bP^m) \cong \bigcup_{[\cL] \in \Pic^\alpha(X)} \Gammahol(\cL^{\oplus m+1} - 0) / \bC^\times
\]
where $\Pic^\alpha(X)$ parameterises isomorphism classes of holomorphic line bundles with first Chern class $\alpha$. The point of this paper is to leverage the homotopy equivalences $\Map(X,\bC^\times) / \bC^\times \simeq H^1(X;\bZ)$ and $\Pic^\alpha(X) \simeq K(H^1(X;\bZ),1)$, and show that they ``cancel each other out'' to produce \cref{thm:main-intro}. This strategy was suggested to us by Phil Tosteson.

\subsection{Outline}

In \cref{sec:point-set} we conscientiously give a point-set model for the space of holomorphic maps, which makes precise the fact that they are given in homogeneous coordinates by sections of line bundles. In \cref{sec:interpolating} we define the number $d(\alpha)$ appearing in the range of \cref{thm:main-intro} and explain how to estimate it in some cases. Finally we prove our main theorem in \cref{sec:proof}.

\setcounter{tocdepth}{1}
\tableofcontents

\subsection*{Acknowledgements} 
I would like to thank Phil Tosteson for making me aware of subtleties concerning holomorphic line bundles and suggesting the strategy used in this paper. This paper grew out of a remark from him on an earlier version of \cite{aumonier_h-principle_2021}. The microfibration trick is derived from joint work with Ronno Das, whom I thank. The arguments in this paper were obtained when I was in Copenhagen, and I thank S\o ren Galatius for very helpful discussions over the years. I was supported by the Danish National Research Foundation through the Copenhagen Centre for Geometry and Topology (DNRF151), as well the ERC under the European Union’s Horizon 2020 research and innovation programme (grant agreement No. 756444) during the writing of this article.

\section{Mapping spaces and point-set topology}\label{sec:point-set}

In this section, we recall well-known facts about mapping spaces to projective space. In particular, we carefully give a point-set model for the space of holomorphic maps in \cref{prop:scheme-isomorphism}. Although we claim no originality for this section, we have chosen to include it to set the remainder of the paper on firm topological ground.

In this part only, we shall use the language of algebraic geometry. By convention we simply say ``scheme'' to mean a complex scheme of finite type, and denote their category by $\catSch$.

We denote by $\Pic^\alpha(X)$ the the connected component of the Picard scheme of $X$ whose complex points are isomorphism classes of line bundles on $X$ with first Chern class $\alpha$. We choose once and for all a Poincar\'e line bundle on $X \times \Pic^\alpha(X)$, see e.g. \cite[Exercise~4.3]{kleiman_picard_2005}, which we denote by $\cP$. Whenever we talk about a line bundle $\cL$ representing an isomorphism class $[\cL] \in \Pic^\alpha(X)$, we shall always mean our chosen representative line bundle $\cL = \iota^*\cP$ on $X \cong X \times \{[\cL]\} \overset{\iota}{\hookrightarrow} X \times \Pic^\alpha(X)$. This applies in particular to make sense of sections of $[\cL]$, i.e. we mean sections of $\cL = \iota^*\cP$. As a shorthand, we write $\cE \defeq \cP^{\oplus m+1}$, and let $p \colon X \times \Pic^\alpha(X) \to \Pic^\alpha(X)$ be the second projection. We shall further assume that
\begin{enumerate}[label=(A\arabic*)]
    \item\label{enum:cohomology-assumption} the cohomology group $H^1(X,\cL)$ vanishes for any $[\cL] \in \Pic^\alpha(X)$;
    \item\label{enum:counit-assumption} the counit (i.e. the evaluation morphism) $\epsilon \colon p^*p_* \cP \to \cP$ is surjective.
\end{enumerate}

\bigskip

By definition, an algebraic map from $X$ to $\bP^m$ is a complex point of the Hom scheme $\Hom(X,\bP^m)$, which exists as a complex quasi-projective scheme by a general representability theorem of Grothendieck \cite[Section~4.c]{grothendieck_techniques_1961}. By Chow's theorem, holomorphic maps from $X$ to $\bP^m$ are algebraic, whence the equality of sets
\[
    \Hol(X,\bP^m) = \Hom(X,\bP^m)(\bC).
\]
The left-hand space is more naturally topologised as a subspace of the continuous mapping space $\Map(X,\bP^m)$ with the compact-open topology, whereas the right-hand space is equipped with the analytic topology. Fortunately these topologies agree, see e.g. \cite[Section~10.4 Th\'eor\`eme~2]{douady_probleme_1966}.

Below, we give a convenient point-set model for the space of algebraic maps. From the universal property of projective space \stacks{01ND}, we learn that the Hom scheme $\Hom(X,\bP^m)$ represents the functor
\begin{align}\label{eqn:hom-scheme-functor}
\begin{split}
    (\catSch)^\opp &\lra \catSet \\
    T &\longmapsto \left\{ (\cL,s_0,\dotsc,s_m)\ \middle\vert \begin{array}{l}
    \cL \text{ is a line bundle on } X \times T \text{ and } s_0,\dotsc,s_m\\
    \text{ are global sections of $\cL$ which generate $\cL$}
  \end{array}\right\} \bigg/ \sim
\end{split}
\end{align}
where $(\cL,s_0,\dotsc,s_m) \sim (\cN,t_0,\dotsc,t_n)$ if there exists an isomorphism $\varphi \colon \cL \to \cN$ such that $\varphi(s_i) = t_i$.
This gives a blueprint to build the scheme of algebraic maps: it should be a scheme which amalgamates $(m+1)$-tuples of generating sections, up to scaling, for all isomorphism classes of line bundles. This is precisely what we do below by constructing it as a certain projective bundle over $\Pic^\alpha(X)$.
Under the assumptions~\ref{enum:cohomology-assumption} and \ref{enum:counit-assumption}, by cohomology and base change, the sheaf $p_*\cE$ is locally free and the kernel
\[
    \cK \defeq \ker\left( p^*p_* \cE \to \cE \right)
\]
is also locally free. The injection $\cK \hookrightarrow p^*p_*\cE$ defines a closed subscheme 
\[
    \cZ \defeq \Proj_{X \times \Pic^\alpha(X)}\big(\Sym \cK^\vee\big) \hookrightarrow \Proj_{X \times \Pic^\alpha(X)}\big(\Sym p^*(p_*\cE)^\vee\big) \cong X \times \Proj_{\Pic^\alpha(X)}\big( \Sym \ (p_*\cE)^\vee \big).
\]
The second projection $\mathrm{pr}$ from the right-hand product is closed as $X$ is proper, hence $\mathrm{pr}(\cZ)$ is closed, and we define 
\[
    \Alg_\alpha(X,\bP^m) \defeq \Proj_{\Pic^\alpha(X)}\big( \Sym \ (p_*\cE)^\vee \big) - \mathrm{pr}(\cZ),
\]
as an open subscheme of the projective bundle $\bP(p_*\cE) \defeq \Proj_{\Pic^\alpha(X)}\big( \Sym\ (p_*\cE)^\vee \big)$. Its complex points form the set
\begin{equation}\label{eqn:complex-points-model}
    \Alg_\alpha(X,\bP^m)(\bC) = \set{([\cL],[s]) \in \Pic^\alpha(X) \times \bP\Gamma(X, \cL^{\oplus m+1})}{\forall x \in X,\ s(x) \neq 0}
\end{equation}
where we recall that $\cL$ denotes the line bundle on $X$ representing the isomorphism class $[\cL]$ afforded by the choice of $\cP$. This is our convenient point-set model:
\begin{proposition}\label{prop:scheme-isomorphism}
The complex schemes $\Hom_\alpha(X,\bP^m)$ and $\Alg_\alpha(X,\bP^m)$ are isomorphic.
\end{proposition}

We shall shortly prove the proposition, but first record a direct consequence of the definitions:
\begin{lemma}\label{lemma:functor-of-points}
Let $T \in \catSch$ be a scheme, and let $q \colon X \times T \to T$ be the projection. Consider triples $(f,\cA,\varphi)$ with
\begin{enumerate}[label=(\roman*)]
    \item a morphism $f \colon T \to \Pic^\alpha(X)$,
    \item a line bundle $\cA$ on $T$,
    \item a surjection $\varphi \colon f^* (p_*\cE)^\vee \twoheadrightarrow \cA$,
\end{enumerate}
such that the composition
\begin{equation}\label{eqn:surjectivity-condition}
    (1 \times f)^*\cE^\vee \overset{\epsilon^\vee}{\lra} (1 \times f)^*p^*( p_*\cE)^\vee = q^*f^* (p_*\cE)^\vee \overset{q^*\varphi}{\lra} q^*\cA
\end{equation}
is a surjective. Two triples $(f,\cA,\varphi)$ and $(f,\cB,\psi)$ are said equivalent if $f = g$, and there exists an isomorphism $\iota \colon \cA \to \cB$ such that $\psi = \iota \circ \varphi$. The scheme $\Alg_\alpha(X,\bP^m)$ represents the functor
\[
    T \longmapsto \left\{ \text{equivalence classes of triples } (f,\cA,\varphi) \text{ as above} \right\}.
\]
\end{lemma}

\begin{proof}
From the universal property of the relative Proj construction, see e.g. \stacks{01NS}, we see that a morphism $g \colon T \to \bP(p_*\cE)$ is exactly given by the data of (i), (ii) and (iii), up to equivalence. The morphism $g$ factors through $\Alg_\alpha(X,\bP^m)$ if and only if the image of $1 \times g \colon X \times T \to X \times \bP(p_*\cE)$ is disjoint from the set $\cZ$. This condition can be checked at each point separately, and thus is equivalent to: for any point $y \colon \Spec(\kappa(y)) \to X \times T$, the morphism $(1 \times f) \circ y$ to $X \times \bP(p_*\cE)$ cannot be lifted to $\cZ$. In other words, transcribing into the functor of points perspective, there is no dotted lift making the following diagram commute:
\[
\begin{tikzcd}[row sep=15pt, column sep=small]
0 \arrow[r] & y^* (1\times f)^* \cE^\vee \arrow[r] & y^*(1 \times f)^* p^*(p_*\cE)^\vee = y^*q^*f^*(p_*\cE)^\vee \arrow[r] \arrow[d, two heads] & y^*(1\times f)^*\cK^\vee \arrow[r] \arrow[ld, dashed] & 0 \\
    & & y^*q^*\cA      
\end{tikzcd}
\]
Using the exactness of the first row, this is equivalent to the composite $y^* (1\times f)^* \cE^\vee \to y^*q^*\cA$ not vanishing. But this is a morphism of $\kappa(y)$-vector spaces, with codomain of dimension 1. Hence it is non-zero if and only if it is surjective. By Nakayama's lemma, this is equivalent to a surjection at the level of stalks (as opposed to fibres). This is precisely the condition~\eqref{eqn:surjectivity-condition} stated in the lemma.
\end{proof}

\begin{proof}[Proof of \cref{prop:scheme-isomorphism}]
We shall show that the functors they represent are naturally isomorphic. Let $T \in \catSch$ be a test scheme, write $q \colon X \times T \to T$ for the projection, and let $(f,\cA,\varphi) \in \Alg_\alpha(X,\bP^m)(T)$ be as in \cref{lemma:functor-of-points}. We shall naturally produce an associated $T$-point of $\Hom_\alpha(X,\bP^m)$ as follows. By dualising \cref{eqn:surjectivity-condition} and tensoring with $(1\times f)^*\cP^\vee$ we obtain an injection
\[
    \cL^\vee \defeq q^*\cA^\vee \otimes (1\times f)^*\cP^\vee \hookrightarrow \cO_{X \times T}^{\oplus m+1},
\]
whose dual gives a point $[\cO_{X \times T}^{\oplus m+1} \twoheadrightarrow \cL] \in \Hom_\alpha(X,\bP^m)(T)$ using \cref{eqn:hom-scheme-functor}. 
Conversely, let $\cO_{X \times T}^{\oplus m+1} \twoheadrightarrow \cL$ be given. By \cite[Exercise~4.3]{kleiman_picard_2005}, there exists a unique morphism $f \colon T \to \Pic^\alpha(X)$ such that
\[
    (1\times f)^*\cP \otimes q^*\cA = \cL
\]
for some line bundle $\cA$ on $T$. Plugging that formula into the given surjection, then tensoring by $(1\times f)^*\cP^\vee$ and finally dualising yields an injection
\[
    q^*\cA^\vee \hookrightarrow (1\times f)^*\cE.
\]
We can then apply the pushforward functor $q_*$ to obtain
\[
    q_*q^*\cA^\vee \cong \cA^\vee \hookrightarrow q_*(1\times f)^*\cE \cong f^* p_*\cE,
\]
where the isomorphisms follow from cohomology and base change, the second one using our assumption~\ref{enum:counit-assumption}. Dualising yields surjection $f^* (p_*\cE)^\vee \twoheadrightarrow \cA$ and thus a $T$-point of $\Alg_\alpha(X,\bP^m)$. One checks directly that global generation of $\cL$ corresponds under this construction to the condition of \cref{eqn:surjectivity-condition}.

Finally, it is straightforward to verify that these constructions are well-defined, are functorial, and are inverse to each other.
\end{proof}

\section{Interpolating line bundles}\label{sec:interpolating}

In this section, we define the number $d(\alpha)$ appearing in the range of \cref{thm:main-intro}.

\begin{definition}\label{def:interpolating}
Let $\cL$ be a line bundle on $X$ and $k \geq 0$ be an integer. We say that $\cL$ is \emph{$k$-interpolating} if for any choice of $k+1$ distinct points $x_0,\dotsc,x_k \in X$ the evaluation map
\[
    \Gammahol(\cL) \lra \bigoplus_{i=0}^k \cL|_{x_i}, \quad s \longmapsto (s(x_0),\dotsc,s(x_k))
\]
is surjective.
\end{definition}

\begin{example}\label{example:interpolating}
A line bundle $\cL$ is $0$-interpolating if and only if it is globally generated, or in other words basepoint free. If $\cL$ is very ample, then it is also $1$-interpolating. More generally, if $\cL$ is $k$-very ample then it is also $k$-interpolating. In particular $\cL^{\otimes k}$ is $k$-interpolating if $\cL$ is very ample.
\end{example}

In practice, it is useful to notice that the property of being $k$-interpolating is implied by a stronger, but easier, cohomological criterion. Indeed, if for any closed subscheme $Z \subset X$ of length $k+1$ with ideal sheaf $\cI_Z$ the coherent cohomology group $H^1(X, \cL \otimes \cI_Z)$ vanishes, then the long exact sequence associated to $0 \to \cI_Z \to \cO_X \to \cO_Z \to 0$ shows that $\cL$ is $k$-interpolating.

\bigskip

Holomorphic maps $X \to \bP^m$ can only have degree in the N\'eron--Severi group $\NS(X) \subset H^2(X;\bZ)$ generated by Chern classes of holomorphic line bundles. We thus make the following:

\begin{definition}\label{def:dalpha-range}
Let $\alpha \in \NS(X)$. We define $d(\alpha) \in \bN \cup \{-\infty\}$ as
\[
    d(\alpha) \defeq \sup \ \set{k \in \bN}{\forall [\cL] \in \Pic^\alpha(X), \ \cL \text{ is $k$-interpolating}}.
\]
When the set of the right-hand side is empty $d(\alpha) = - \infty$, but we will never consider this situation in this article. By convention we also set $d(\alpha) = - \infty$ if $\alpha \notin \NS(X)$.
\end{definition}

We are not aware of a general formula for $d(\alpha)$ valid for any smooth complex projective variety $X$. In fact, it seems quite unlikely that such a formula should exist, see \cite[Appendix]{aumonier_scanning_2023} for a related discussion. Nonetheless, $d(\alpha)$ can be computed in particular cases of interest. The following two examples show that \cref{thm:main-intro} recovers the theorems of Segal \cite{segal_topology_1979} and Mostovoy \cite{mostovoy_spaces_2006}.

\begin{lemma}
If $X$ is a Riemann surface of genus $g$, then $\Pic^\alpha(X)$ parameterises isomorphism classes of line bundles of degree $\alpha \in \bZ \cong H^2(X;\bZ)$. If $\alpha \geq 2g$ then $d(\alpha) \geq \alpha - 2g$.
\end{lemma}
\begin{proof}
This is a direct consequence of the cohomological criterion describe above and the Riemann--Roch theorem, see e.g. \cite[Lemma~A.1]{aumonier_scanning_2023}.
\end{proof}

\begin{lemma}
The line bundle $\cO(\alpha)$ on $X = \bP^n$ is $\alpha$-interpolating when $\alpha \geq 0$. In that case $d(\alpha) \geq \alpha$.
\end{lemma}
\begin{proof}
The line bundle $\cO(1)$ is very ample, so $\cO(\alpha) = \cO(1)^{\otimes \alpha}$ is $\alpha$-interpolating by \cref{example:interpolating}.
\end{proof}

Let us also mention one more example for which our main theorem is new:
\begin{lemma}
Let $X = \bP^{n_1} \times \cdots \times \bP^{n_r}$ be a product of projective spaces of dimension $\sum_i n_i \leq m$, with $n_i \geq 1$. Then the line bundle $\cO_X(\alpha_1,\dotsc,\alpha_r)$ is $\min(\alpha_1,\dotsc,\alpha_r)$-interpolating. In other words, $d(\alpha_1,\dotsc,\alpha_r) \geq \min(\alpha_1,\dotsc,\alpha_r)$.
\end{lemma}
\begin{proof}
A line bundle on $X$ has the form $\cO_X(k_1,\dotsc,k_r)$ for some $k_i \in \bZ$. It is globally generated, resp. very ample, if $k_i \geq 0$, resp. $k_i > 0$, for all $i$. The result follows from \cref{example:interpolating}.
\end{proof}

Although we give no explicit formula for $d(\alpha)$, it can be shown that the function $d(-)$ is unbounded. In particular, the degree range of \cref{thm:main-intro} can be arbitrarily large if one is willing to work with maps of ``ample enough'' degree. More precisely, we have the following qualitative result: 

\begin{lemma}
Let $\alpha, \beta \in \NS(X)$ with $\beta $ ample. Then for any $M \geq 0$ there exists $k_0 \geq 0$ such that $d(\alpha + k\beta) \geq M$ for all $k \geq k_0$.
\end{lemma}
\begin{proof}[Proof sketch]
This can be shown using the properness of the Hilbert scheme of 0-dimensional subschemes $Z\subset X$ of a given finite length, combined with the cohomological criterion given above. See \cite[Proposition~A.7]{aumonier_scanning_2023}. 
\end{proof}

\section{Proof of \texorpdfstring{\cref{thm:main-intro}}{Theorem~\ref{thm:main-intro}}}\label{sec:proof}

In this section we prove our main theorem. We use the point-set model of the space of homolorphic maps given in \cref{prop:scheme-isomorphism}, which we implicitly regard as the topological space defined by the analytic topology on the complex points~\eqref{eqn:complex-points-model}. In particular, we fix a class $\alpha \in \NS(X)$ such that 
\begin{enumerate}
    \item $d(\alpha) \geq 0$, so that assumption~\ref{enum:counit-assumption} is satisfied;
    \item and $\alpha - c_1(K_X)$ is ample, so that assumption~\ref{enum:cohomology-assumption} is satisfied by the Kodaira vanishing theorem.
\end{enumerate}

\subsection{An intermediate model for continuous maps}

We construct an analogue of $\Alg_\alpha(X,\bP^m)$ for continuous maps by replacing holomorphic sections by continuous ones:
\begin{definition}\label{def:fibrewise-continuous}
The space of \emph{non-vanishing fibrewise continuous sections of $\cE$ over $\Pic^\alpha(X)$} is
\[
    \Gammaconfib(\cE-0 \rel \Pic^\alpha(X)) \defeq \set{([\cL], s)}{[\cL] \in \Pic^\alpha(X), \ s \in \Gammacon(\cL^{\oplus m+1} - 0)},
\]
where $\cL$ denotes the representative of the equivalence class $[\cL]$ selected by the choice of $\cP$. It is topologised as a subspace
\[
    \Gammaconfib(\cE-0 \rel \Pic^\alpha(X)) \hookrightarrow \Map(X, \cE - 0), \quad ([\cL], s) \longmapsto (x \mapsto s(x))
\]
of the mapping space. The group $\bC^\times$ acts on $\cE - 0$ and thus on sections by post-composition, and we denote by
\[
    \bP \Gammaconfib(\cE-0 \rel \Pic^\alpha(X)) \defeq \Gammaconfib(\cE-0 \rel \Pic^\alpha(X)) / \bC^\times
\]
the quotient by that group action.
\end{definition}

\begin{remark}
Associated to the vector bundle $\cE$ are two sheaves: the sheaf of algebraic sections, also denoted $\cE$, and the sheaf of continuous sections $\cE^\mathrm{top}$. Then $p_*\cE^\mathrm{top}$ is the sheaf of fibrewise continuous sections of $\cE$, and thus \cref{def:fibrewise-continuous} mirrors our definition of $\Alg_\alpha(X,\bP^m)$.
\end{remark}

The projection $\cE - 0 \to X \times \Pic^\alpha(X) \to \Pic^\alpha(X)$ induces a continuous map under applying the functor $\Map(X,-)$. By our definition of the topology, this implies that the projection 
\begin{equation}\label{eqn:projection-bundle}
    \bP \Gammaconfib(\cE-0 \rel \Pic^\alpha(X)) \to \Pic^\alpha(X), \quad ([\cL], s) \longmapsto [\cL]
\end{equation}
is continuous. In fact, we have the following:
\begin{lemma}\label{lemma:fibration}
The projection~\eqref{eqn:projection-bundle} is a fibre bundle, whose fibre above $[\cL]$ is $\bP \Gammacon(\cL^{\oplus m+1} - 0)$.
\end{lemma}
\begin{proof}
Let $U \subset \Pic^\alpha(X)$ be a small contractible open set, and choose some $[\cL_0] \in U$. A topological vector bundle over a contractible base is trivial, and we thus obtain an isomorphism
\[
    \psi \colon \cP|_{X \times U} \overset{\cong}{\lra} \cL_0 \times U
\]
of complex line bundles over $X \times U$.
Then the homeomorphism
\begin{align*}
    \bP \Gammaconfib(\cE-0 \rel \Pic^\alpha(X))|_U &\overset{\cong}{\lra} U \times \bP \Gammacon(\cL_0^{\oplus m+1} - 0) \\
    ([\cL], [s]) &\longmapsto ([\cL], [\psi \circ s])
\end{align*}
is a local trivialisation.
\end{proof}

By regarding a holomorphic section as continuous, any element of $\Alg_\alpha(X,\bP^m)$ gives a fibrewise section. The following is our main comparison theorem:
\begin{theorem}\label{thm:homology-isomorphism}
The natural inclusion
\[
    \Alg_\alpha(X,\bP^m) \hookrightarrow \bP \Gammaconfib(\cE-0 \rel \Pic^\alpha(X))
\]
induces an isomorphism in integral homology in the range of degrees $* < (2m-2n+1) \cdot d(\alpha) + 2(m-n-1)$.
\end{theorem}

We shall prove this result in \cref{subsec:homology-isoproof}, but first explain in the next section that the space of fibrewise sections is homotopy equivalent to $\Map_\alpha(X,\bP^m)$, thus finishing the proof of \cref{thm:main-intro}.

\subsection{Maps and fibrewise sections}

Let $\cL_\alpha$ be a holomorphic line bundle on $X$ with $c_1(\cL_\alpha) = \alpha$. Recall that up to isomorphism a topological complex line bundle on $X$ is determined by its Chern class. Hence, for any $[\cL] \in \Pic^\alpha(X)$ we may arbitrarily pick an isomorphism $\cL \cong \cL_\alpha$ of topological complex line bundles (again, $\cL$ means the chosen representative afforded by $\cP$). We thus obtain a map
\begin{equation}\label{eqn:projectivisation-sections}
    \Gammacon(\cL^{\oplus m+1} - 0) \overset{\cong}{\lra} \Gammacon(\cL_\alpha^{\oplus m+1} - 0) \lra \Gammacon(\bP(\cL_\alpha^{\oplus m+1})), \quad s \longmapsto \bP s
\end{equation}
which is well-defined and does not depend on the chosen isomorphism. Indeed, any two isomorphisms differ by an action of the automorphism group $\mathrm{Aut}(\cL_\alpha) \cong \Map(X, \bC^\times)$, and thus the projectivisation $\bP s$ of a section $s$ does not depend on the choice of isomorphism $\cL \cong \cL_\alpha$.
As $\cL_\alpha^{\oplus m+1} \cong \cO_X^{\oplus m+1} \otimes \cL_\alpha$ and the projectivisation of a vector bundle is invariant under tensoring by a line bundle, the projective bundle is trivial and we obtain a homeomorphism
\[
    \Gammacon(\bP(\cL^{\oplus m+1})) \cong \Map(X,\bP^m).
\]
By \cite[Lemma~2.5]{crabb_function_1984}, the map~\eqref{eqn:projectivisation-sections} lands inside the connected components of maps of degree $\alpha$. The following is the final piece for the proof of the main result of this paper:
\begin{theorem}\label{thm:homotopy-equivalence}
The projectivisation map
\[
    \Gammaconfib(\cE-0 \rel \Pic^\alpha(X)) \lra \Map_\alpha(X,\bP^m), \quad ([\cL],s) \longmapsto \bP s
\]
is a weak homotopy equivalence.
\end{theorem}

By combining \cref{thm:homology-isomorphism} and \cref{thm:homotopy-equivalence}, we directly obtain \cref{thm:main-intro}.

\begin{proof}[Proof of \cref{thm:homotopy-equivalence}]
We have a commutative square
\begin{equation}\label{eqn:comparison-square}
\begin{tikzcd}
\Gammacon(\cL_\alpha^{\oplus m+1} - 0) \arrow[r, hook] \arrow[d, equal] & \Gammaconfib(\cE-0 \rel \Pic^\alpha(X)) \arrow[d] \\
\Gammacon(\cL_\alpha^{\oplus m+1} - 0) \arrow[r] & {\Map_\alpha(X,\bP^m)}
\end{tikzcd}
\end{equation}
By \cref{lemma:fibration}, the homotopy fibre of the upper horizontal arrow is homotopy equivalent to the loop space $\Omega \Pic^\alpha(X) \simeq H^1(X;\bZ)$. On the other hand, by \cite[Proposition~2.6]{crabb_function_1984} there is a principal $\Map(X,\bC^\times)$-bundle
\[
    \Map(X,\bC^\times) \lra \Gammacon(\cL_\alpha^{\oplus m+1} - 0) \lra \Map_\alpha(X,\bP^m).
\]
Taking the quotient fibrewise by the subgroup $\bC^\times \subset \Map(X,\bC^\times)$ of constant maps gives the principal bundle
\[
    \Map(X,\bC^\times)/\bC^\times \lra \bP\Gammacon(\cL_\alpha^{\oplus m+1} - 0) \lra \Map_\alpha(X,\bP^m).
\]
In particular, the homotopy fibre of lower horizontal arrow of~\eqref{eqn:comparison-square} is $\Map(X,\bC^\times)/\bC^\times \simeq H^1(X;\bZ)$. By the 5-lemma, it therefore suffices to show that the induced map between the horizontal homotopy fibres of~\eqref{eqn:comparison-square} induces a bijection on connected components. This is straightforward, although slightly tedious, using the standard explicit models of the homotopy fibres obtained from the path space fibration. The argument in a closely related setting is entirely spelled out in \cite[Lemma~5.12]{aumonier_scanning_2023}.
\end{proof}

\subsection{The proof of \texorpdfstring{\cref{thm:homology-isomorphism}}{Theorem~\ref{thm:homology-isomorphism}}}\label{subsec:homology-isoproof}

The map of \cref{thm:homology-isomorphism} fits in the following commutative diagram where the top row is its restriction to the fibre above an $[\cL] \in \Pic^\alpha(X)$:
\begin{equation}\label{eqn:comparison-diagram}
\begin{tikzcd}[column sep=small]
\bP \Gammahol(\cL^{\oplus m+1} - 0) \arrow[d] \arrow[r, hook] & \bP \Gammacon(\cL^{\oplus m+1} - 0) \arrow[d]             \\
{\Alg_\alpha(X,\bP^m)} \arrow[d] \arrow[r, hook]         & \bP \Gammaconfib(\cE-0 \rel \Pic^\alpha(X)) \arrow[d] \\
\Pic^\alpha(X) \arrow[r, equal]            & \Pic^\alpha(X)
\end{tikzcd}
\end{equation}
The map between the fibres was studied in \cite{aumonier_h-principle_2021} where a general h-principle for holomorphic sections was established. As a particular case, we have:
\begin{theorem}[Compare {\cite[Theorem~2.13]{aumonier_h-principle_2021}}]\label{thm:hprinciple}
The inclusion
\[
    \bP \Gammahol(\cL^{\oplus m+1} - 0) \hookrightarrow \bP \Gammacon(\cL^{\oplus m+1} - 0)
\]
induces an isomorphism in integral homology the range of degrees $* < (2m-2n+1) \cdot d(\alpha) + 2(m-n)$.
\end{theorem}
\begin{proof}
It is a straightforward application of the main result of \cite{aumonier_h-principle_2021}, but we comment on two points. Firstly, in \cite[Theorem~2.13]{aumonier_h-principle_2021}, the result is stated using the more classical notion of jet ampleness instead of interpolation as in \cref{def:interpolating}. However, its proof only uses the interpolation property of the jet bundle, which in our case (jets of order $0$) translates to the interpolation property of the line bundle $\cL$. Secondly, the result in \textit{loc. cit.} concerns the section spaces before projectivising. But the $\bC^\times$ action is free, so the result for the quotients is directly deduced from a comparison of the Serre spectral sequences of the (homotopy) orbits fibrations.
\end{proof}

If both columns in~\eqref{eqn:comparison-diagram} were fibration sequences, \cref{thm:homology-isomorphism} would follow from \cref{thm:hprinciple} by a straightforward comparison of the associated Serre spectral sequences. Although the right-hand column is a fibration by \cref{lemma:fibration}, the left-hand is only the restriction of the projective bundle $\bP(p_*\cE)$ to the open subset $\Alg_\alpha(X,\bP^m)$. In particular, it is not a fibration in general. The following result, a homological version of a theorem of Weiss \cite[Lemma~2.2]{weiss_what_2005} and Raptis \cite[Theorem~1.3]{raptis_serre_2017} on Serre microfibrations, shows however that this can be leveraged:
\begin{lemma}\label{lemma:micro-comparison}
Let $\pi \colon E \to B$ be a fibre bundle, and $\rho \colon U \to B$ be the restriction of a fibre bundle $E' \to B$ to an open subset $U \subset E'$. Let $f \colon U \to E$ be a map over $B$ and suppose that for every $b \in B$, the restriction to the fibre
\[
	f_b \colon \rho^{-1}(b) \lra \pi^{-1}(b)
\]
is homology $m$-connected \footnote{We say that a map $A \to B$ is homology $m$-connected if it induces an isomorphism on homology groups $H_i(A) \to H_i(B)$ for $i < m$ and a surjection when $i = m$.}. Then $f \colon U \to V$ is homology $m$-connected.
\end{lemma}
\begin{proof}
See \cite[Lemma~3.8]{aumonier_homological_2023}.
\end{proof}

\begin{proof}[Proof of \cref{thm:homology-isomorphism}]
Using the diagram~\eqref{eqn:comparison-diagram} it suffices\footnote{The map of \cref{thm:hprinciple} induces an isomorphism in homology in the range $* < (2m-2n+1) \cdot d(\alpha) + 2(m-n)$, and thus is at least $-2 + [(2m-2n+1) \cdot d(\alpha) + 2(m-n)]$-homology connected. This explains the slightly worse bound.} to apply \cref{lemma:micro-comparison} to the map
\[
    \Alg_\alpha(X,\bP^m) \lra \bP \Gammaconfib(\cE-0 \rel \Pic^\alpha(X)).
\]
The assumptions are verified by \cref{lemma:fibration} and \cref{thm:hprinciple}.
\end{proof}

\printbibliography

\end{document}